\documentclass[12pt]{amsart}
\usepackage{amssymb}

\usepackage{tikz}
\newtheorem{Teo}{Theorem}[section]
\newtheorem{introtheorem}{Theorem}

\newtheorem*{theorem*}{Theorem}
\newtheorem{Lema}[Teo]{Lemma}
\newtheorem{Obs}[Teo]{Remark}

\newtheorem{Exa}[Teo]{Example}
\newtheorem{Def}[Teo]{Definition}

\newtheorem{Cor}[Teo]{Corollary}
\newtheorem{Que}[Teo]{Question}
\newtheorem{Prop}[Teo]{Proposition}

\newcommand{\VR}{\mathcal{O}}

\newtheorem{lem-def}[Teo]{Lemma-Definition}

\renewenvironment{proof}{{\bfseries Proof.}}{\qed}
\def\ars#1{\renewcommand\arraystretch{#1}}

\def\diso{\lower.4ex\hbox{$\downarrow$}\raise.4ex\hbox{\mbox{\scriptsize
$\wr$}}}

\def\ism{\lower.3ex\hbox{\ars{.08}$\begin{array}{c}\,\to\\\mbox{\tiny $\sim\,$}\end{array}$}}
\def\iso{\ \lower.3ex\hbox{\ars{.08}$\begin{array}{c}\lra\\\mbox{\tiny $\sim\,$}\end{array}$}\ }

\def\lg{l\raise.6ex\hbox to.2em{\hss.\hss}l}

\def\lra{\,\longrightarrow\,}

\def\orb{\hbox to  .3em{$\backslash$}\backslash}

\newcounter{cs}
\stepcounter{cs}
\newcommand{\casos}{\begin{itemize}}
\newcommand{\fcasos}{\end{itemize}\setcounter{cs}{1}}

\newfont{\tit}{cmr12 scaled \magstep3}

\title{Ramification ideals for products of pure subgroups}
\makeatletter
\@namedef{subjclassname@2010}{%
  \textup{2010} Mathematics Subject Classification}

\author[Josnei Novacoski]{Josnei Novacoski}
\address{Departamento de Matem\'{a}tica,\newline Instituto de Ci\^encias Matem\'aticas e de Computa\c{c}\~ao,\newline USP, Avenida Trabalhador São-carlense, 400,  CEP 13566--590, S\~ao Carlos--SP, Brazil}
\email{josnei@icmc.usp.br}

\thanks{During the realization of this project the author was supported by a grant from Funda\c{c}\~ao de Amparo \`a Pesquisa do Estado de S\~ao Paulo (process number 2025/10540-0) and a grant from Conselho Nacional de Desenvolvimento Cient\'ifico e Tecnol\'ogico (process number 317185/2025-0). The author thanks Apan Dutta for valuable discussions during the realization of this project.}


\begin{document}

\begin{abstract}
We study ramification ideals of finite Galois extensions of henselian
valued fields whose Galois groups are products of pure subgroups.
We give a criterion showing that the ideals supplied by the factors need
not exhaust all subgroup ideals, and construct $C_p^2$-counterexamples
of rank one with each defect $1,p,p^2$. For ramification-adapted
families, whose intersections generate each ramification group, every
subgroup ideal is represented by a cyclic subgroup of one member; when
the members are pure, this gives a description in terms of distance sets.
Such adapted decompositions always exist for elementary abelian groups.
We also characterize principality of ramification ideals through
defectlessness of the pure adapted factors. Kuhlmann's example shows that
these conditions do not imply defectlessness of the full extension.
\end{abstract}

\maketitle
\enlargethispage{2pt}

\section{Introduction}
Let $\mathcal E=(L/K,v)$ be a finite Galois extension of henselian valued
fields and set $\mathcal G={\rm Gal}(L/K)$. 
For $\sigma\in\mathcal G$ we
consider the ideal
\[
I_\sigma=\left\{c\in L\mid vc\geq
v\left(\frac{\sigma b-b}{b}\right)\mbox{ for some }b\in L^\times\right\},
\]
and for a subgroup $H$ of $\mathcal G$ we set
\[
I_H=\bigcup_{\sigma\in H}I_\sigma.
\]
Write $\VR_L$ for the valuation ring and $\mathcal M_L$ for its maximal
ideal. We are interested in the sets
\[
\mathcal I(\mathcal E)=\{I_H\mid \{1\}\neq H\leq\mathcal G\}
\quad\mbox{and}\quad
{\rm Ram}(\mathcal E)=
\{I\in\mathcal I(\mathcal E)\mid 0\neq I\subseteq\mathcal M_L\}.
\]
We call the elements of ${\rm Ram}(\mathcal E)$ the ramification ideals
of $\mathcal E$, following \cite{Topics}. In particular, a subgroup ideal equal to the whole valuation
ring is excluded from ${\rm Ram}(\mathcal E)$. Different subgroups can give the
same ideal. Our goal is to describe these ideals using distances of
algebraic elements and to study their relation with defect.

For an intermediate field $F$ and $\theta\in L\setminus F$, denote
\[
D_1(\theta,F)=\{v(\theta-c)\mid c\in F\}.
\]
For a nonempty subset $S$ of $vL$ bounded from below, set
\[
I_S=\{b\in L\mid vb\geq s\mbox{ for some }s\in S\}.
\]
We say that $\theta$ is \textbf{pure} over $F$ if it has depth one, and
that $(L/F,v)$ is pure if it admits a pure generator; see Section
\ref{notation}. The following formula describes the ramification ideals
when $(L/K,v)$ is pure. We state it in Theorem \ref{pureextens} as a reformulation of
\cite[Lemma 8.8 and Theorem 8.14]{Dutta}.

\begin{introtheorem}\label{intro:pure}
Suppose that $(L/K,v)$ is pure on $\theta$. For every nontrivial subgroup
$H\leq\mathcal G$, we have
\[
I_H=I_{\beta_H-D_1(\theta,K)},
\qquad
\beta_H=\min_{\sigma\in H\setminus\{1\}}v(\sigma\theta-\theta).
\]
\end{introtheorem}

For $H\leq\mathcal G$, let $K_H=L^H$ be its fixed field. We call $H$
\textbf{pure} if $(L/K_H,v)$ is pure. Every subgroup of prime order is pure:
this follows from the divisibility of Okutsu degrees in
\cite[Corollary 1.18(1)--(2) and Theorem 3.4]{NartOkutsu}; see Lemma
\ref{lem:prime-pure}. Suppose now that
$\mathcal G=H_1\cdots H_r$ is a product of pure subgroups. Choose a pure
generator $\alpha_i$ of $(L/K_{H_i},v)$ and put
\[
\mathbf D_i=D_1(\alpha_i,K_{H_i}).
\]
Write $\gamma_{i1},\ldots,\gamma_{is_i}$ for the distinct values of
$v(\sigma\alpha_i-\alpha_i)$, with $\sigma\in H_i\setminus\{1\}$.
Theorem \ref{intro:pure} suggests asking whether
\[
\mathcal I(\mathcal E)=
\{I_{\gamma_{ij}-\mathbf D_i}\mid 1\leq i\leq r,\ 1\leq j\leq s_i\}.
\]
In Section \ref{sec:productpure} we study this question. Our first result
is the following bound, proved in Theorem \ref{mainresultsofar}.

\begin{introtheorem}\label{intro:bound}
With the notation above, take a nontrivial subgroup $H\leq\mathcal G$.
We say that $\sigma_i\in H_i$ appears in $H$ if there exist
$\sigma_j\in H_j$, $j\neq i$, such that
$\sigma_1\cdots\sigma_r\in H$. For each $i$ for which a nontrivial
element of $H_i$ appears in $H$, set
\[
\gamma_i=\min\{v(\sigma_i\alpha_i-\alpha_i)\mid
\sigma_i\in H_i\setminus\{1\}\mbox{ appears in }H\}.
\]
Then, taking the minimum over these indices,
\[
I_H\subseteq I_{\gamma_H},
\qquad \gamma_H=\min_i\{\gamma_i-\mathbf D_i\}.
\]
If there is an index $i$ such that
$\gamma_i-\mathbf D_i<\gamma_j-\mathbf D_j$ for every other such index
$j$, then $I_H=I_{\gamma_H}$.
\end{introtheorem}

Here the differences and their order are understood as final segments,
as in Section \ref{notation}. The equality can fail when the minimum
occurs in two factors. The following criterion explains this failure
(Theorem \ref{teo:counterexample-criterion}).

\begin{introtheorem}\label{intro:counterexample-criterion}
Suppose that $\mathcal G=\langle\sigma\rangle\times\langle\tau\rangle
\simeq C_p^2$ and $I_\tau\subsetneq I_\sigma$. Then
\[
\mathcal G=\langle\sigma\rangle\times\langle\sigma\tau\rangle,
\qquad I_{\sigma\tau}=I_\sigma=I_{\mathcal G}.
\]
Both new factors are pure, and this decomposition does not supply
$I_\tau$.
For every prime $p$ and every $d\in\{1,p,p^2\}$, there is such a
counterexample over a henselian rank-one valued field of characteristic
$p$, with $d(L/K,v)=d$.
\end{introtheorem}

For a compositum $L=K(\alpha,\beta)$ of two linearly disjoint
Artin--Schreier extensions, let $\sigma$ translate $\alpha$ by $1$ and fix
$\beta$, and let $\tau$ translate $\beta$ by $1$ and fix $\alpha$.
Their ideals are determined by the relative distances
$D_1(\alpha,K(\beta))$ and $D_1(\beta,K(\alpha))$.
The example in \cite{Josnei} with $\theta=\alpha+t\beta$ has relative
cuts $0^-$ and $-1^-$ and gives
\[
I_\tau=\{b\in L:vb>1\}\subsetneq\mathcal M_L=I_\sigma.
\]
This gives the immediate case $d=p^2$ in Example \ref{exa:counterexample}.
Example \ref{exa:counterexample-defect-p} realizes $d=p$ by adjoining an
Artin--Schreier root to a field with an irrational Gauss valuation, while
Example \ref{exa:counterexample-defectless} realizes $d=1$ over $k((t))$.
All three valuations have rank one; see Corollary
\ref{cor:all-defects-counterexamples}. In contrast, Kuhlmann's example
in \cite[Section 3.5]{Topics} does not satisfy the hypotheses of
Theorem \ref{intro:counterexample-criterion}: it has $d=p$, but all
nontrivial subgroup ideals coincide, so no choice of generators can
satisfy the required strict inclusion $I_\tau\subsetneq I_\sigma$.
Thus it is not a counterexample to Question \ref{ques:pure-factors},
and the defect alone does not determine whether the criterion applies.

This leads us to consider decompositions compatible with the ramification
groups. For an $\VR_L$-ideal $I$, set
\[
\mathcal G_I=\left\{\sigma\in\mathcal G\mid
\frac{\sigma b-b}{b}\in I\mbox{ for every }b\in L^\times\right\}.
\]
As $I$ varies, these groups form the ramification filtration. Write
$U=\mathcal G_I$ for any one of its groups. A family of nontrivial
subgroups $H_1,\ldots,H_r$ is called \textbf{ramification-adapted} if
\[
U=\langle U\cap H_1,\ldots,U\cap H_r\rangle
\quad\mbox{for every }U=\mathcal G_I.
\]
The brackets denote the subgroup generated by the intersections: every
element of $U$ must be a product of elements from the chosen subgroups
which themselves belong to $U$. The condition is imposed on each group
of the ramification filtration. A decomposition is adapted when its
factors satisfy this condition. For an internal direct product, it says
\[
U=(U\cap H_1)\times\cdots\times(U\cap H_r).
\]
Thus the decomposition of an element of $U$ into components stays inside
$U$. In the elementary abelian case, this means that every ramification
group is spanned by a subset of the chosen basis.
For these families we obtain the following result
(Theorem \ref{teo:adapted}).

\begin{introtheorem}\label{intro:adapted}
Let $H_1,\ldots,H_r$ be a ramification-adapted family. For every
nontrivial subgroup $H\leq\mathcal G$, there exist an index $i$ and a
nontrivial cyclic subgroup $Q\leq H_i$ such that
\[
I_H=I_Q.
\]
If every $H_i$ is pure, with the data $\alpha_i,\mathbf D_i,\gamma_{ij}$
as above, then
\[
I_H=I_Q=I_{\gamma_{ij}-\mathbf D_i}
\]
for some $i,j$. If every $H_i$ has prime order, then $Q=H_i$.
\end{introtheorem}

If $\mathcal G=H_1\times\cdots\times H_r$ and every $\mathcal G_I$ is a
product of some of these factors, we can take $I_H=I_{H_i}$ even when the
factors have larger order; see Corollary \ref{cor:block-adapted}.
For elementary abelian groups, adapted decompositions always exist.

\begin{introtheorem}\label{intro:elementary}
Suppose that $\mathcal G\simeq C_p^r$. There is a ramification-adapted
decomposition
\[
\mathcal G=H_1\times\cdots\times H_r,\qquad |H_i|=p.
\]
For any such decomposition,
\[
\mathcal I(\mathcal E)=\{I_{H_1},\ldots,I_{H_r}\}.
\]
\end{introtheorem}

This follows by choosing a basis adapted to the flag of ramification
groups (Corollary \ref{cor:adaptedexists}). All factors are pure by
Lemma \ref{lem:prime-pure}.

Defectless extensions have principal ramification ideals
\cite[Proposition 3.1]{Topics}. This raises the converse question.
\begin{Que}\label{Quesaboutramif}
Under what conditions does principality of all ramification ideals
imply that $L/K$ is defectless?
\end{Que}
For pure extensions the answer is positive; see Corollary
\ref{princvsdeft}. Adapted families give a criterion in terms of the
defects of their fixed-field extensions. For elementary abelian groups,
this yields a characterization involving all top degree-$p$ extensions,
without choosing a decomposition (Corollaries
\ref{cor:pure-adapted-principality} and
\ref{cor:top-degree-p-principality}).

\begin{introtheorem}\label{intro:principality}
Let $H_1,\ldots,H_r$ be a ramification-adapted family of pure subgroups.
Then all ramification ideals of $\mathcal E$ are principal if and only if
$L/K_{H_i}$ is defectless for every $i$.

If $\mathcal G\simeq C_p^r$, then all ramification ideals of $\mathcal E$
are principal if and only if $(L/F,v)$ is defectless for every
intermediate field $F$ with $[L:F]=p$.
\end{introtheorem}

For the elementary abelian case, it suffices to check the top extensions
associated with any one ramification-adapted decomposition, by Corollary
\ref{cor:adaptedprincipality}. Theorem \ref{intro:elementary} guarantees
that such a decomposition exists.

Kuhlmann's degree-$p^2$ example from \cite[Section 3.5]{Topics} illustrates
precisely this distinction between the top extensions and the full
extension. Its Galois group is $C_p^2$, and all nontrivial subgroup ideals
coincide with one principal ideal. By the elementary abelian criterion
in Theorem \ref{intro:principality}, equivalently, every intermediate
field $F$ with $[L:F]=p$ satisfies $d(L/F,v)=1$. Nevertheless,
$d(L/K,v)=p$. Thus even defectlessness of every top degree-$p$ extension
does not imply defectlessness of $L/K$. We explain this example in
Subsection \ref{subsec:kuhlmann}; it shows why Question
\ref{Quesaboutramif} requires additional hypotheses.

Section \ref{sec:pure-ramification} collects the notation and the
ideal-indexed ramification correspondence, establishes the formula for pure
extensions, and derives the principality criteria, including the prime-degree
case. Finally, Section \ref{sec:ASapplications} gives applications to composita
of Artin--Schreier extensions, including formulas for their coordinate
ideals and the independent defect examples considered in \cite{Josnei}.

Appendix \ref{app:AS-galois} gives a recursive criterion for an
Artin--Schreier tower to be Galois at every stage, together with examples.
It checks the Galois structure of examples from \cite{Dutta,Nart},
including the cyclic quartic tower used in Appendix
\ref{app:pure-products}. The latter appendix illustrates the product-of-pure-subgroups
hypothesis of Section \ref{sec:productpure}: it holds for elementary
abelian extensions, while two rank-one cyclic quartic examples show that
it can hold or fail for the same abstract Galois group.

\section{Ramification ideals and pure extensions}\label{sec:pure-ramification}
\subsection{Notation and terminology}\label{notation}
We retain the notation of the introduction. All algebraic extensions
are taken in a fixed valued algebraic closure $(\overline K,v)$.
An algebraic extension is \textbf{unibranched} if the valuation extends
uniquely. We write $F^h$ for the henselization of a valued field $F$
in this closure; $F$ is \textbf{henselian} if $F=F^h$.
For a finite extension $\mathcal E=(L/K,v)$, write
$e(L/K)=(vL:vK)$ and $f(L/K)=[Lv:Kv]$. Its \textbf{defect} is
\[
d(\mathcal E)=\frac{[L^h:K^h]}{[Lv:Kv](vL:vK)}.
\]
The extension is \textbf{defectless} if $d(\mathcal E)=1$, and
\textbf{immediate} if $vL=vK$ and $Lv=Kv$.

An element $\theta$ is \textbf{pure} over $F$ if it has depth one,
and $(L/F,v)$ is \textbf{pure on $\theta$} if $L=F(\theta)$ and $\theta$
is pure over $F$. We use depth in the sense of Okutsu sequences;
see \cite[Section 3]{NartOkutsu} and \cite[Section 2.5]{Dutta}.

A subset $S\subseteq vL$ is a \textbf{final segment} if
$s\in S$ and $\gamma\geq s$ imply $\gamma\in S$. We order final
segments by reverse inclusion:
\[
S_1\leq S_2\quad\Longleftrightarrow\quad S_2\subseteq S_1.
\]
For $\gamma\in vL$, write $\gamma^+=\{\delta\in vL\mid\delta\geq\gamma\}$.
Thus $\gamma\leq\lambda$ if and only if $\gamma^+\leq\lambda^+$.
For a nonempty subset $A\subseteq vL$ bounded from below, put
\[
I_A=\{0\}\cup\{b\in L^\times\mid vb\geq a\mbox{ for some }a\in A\}.
\]
This agrees with $I_A=\{0\}\cup\{b\in L^\times\mid vb\in A\}$ when
$A$ is a final segment. We also set $I_\varnothing=(0)$.

For $K\subseteq F\subseteq L$ and $\theta\in L\setminus F$, write
\[
D_1(\theta,F)=\{v(\theta-c)\mid c\in F\}.
\]
This need not be an initial segment of $vL$. The \emph{relative distance cut}
$d_1(\theta,F)$ has as its lower set the initial segment generated by
$D_1(\theta,F)$ in $vL$. For $s\in vL$, we use
$s-D_1(\theta,F)$ for the final segment
\[
\{\gamma\in vL\mid\gamma\geq s-v(\theta-c)
\mbox{ for some }c\in F\}.
\]
In particular, $-D_1(\theta,F)$ means $0-D_1(\theta,F)$.
For a final segment $S$, the notation $\gamma\geq S$ means
$\gamma\in S$ or $\gamma=\infty$.

An \textbf{Artin--Schreier (AS) extension} is an extension $F(\alpha)/F$
of degree $p=\operatorname{char}(F)>0$ with $\alpha^p-\alpha\in F$;
we call $\alpha$ an \textbf{AS element} over $F$.
A family of finite extensions $L_1,\ldots,L_n$ is \textbf{linearly disjoint}
over $K$ if
\begin{equation}\label{linearlydisj}
[L_1\cdots L_n:K]=\prod_{i=1}^n[L_i:K].
\end{equation}
Equivalently, each $L_i$ is linearly disjoint from the compositum of the
others. When all $L_i/K$ are Galois, this is equivalent to
\[
L_i\cap\Bigl(\prod_{j\ne i}L_j\Bigr)=K\qquad\mbox{for every }i.
\]
We call algebraic elements linearly disjoint when the extensions they
generate are linearly disjoint.

\subsection{Ramification groups indexed by ideals}\label{subsec:idealramgroups}
We use the ideal-indexed formulation from
\cite[Section 2.5]{Topics}. Since $(K,v)$ is henselian, the decomposition
group of $L/K$ is the whole group $\mathcal G$. For an $\VR_L$-ideal $I$
set
\[
\mathcal G_I:=\left\{\sigma\in\mathcal G\ \middle|\
\frac{\sigma b-b}{b}\in I\mbox{ for every }b\in L^\times\right\}.
\]
The groups $\mathcal G_I$ are linearly ordered by inclusion as $I$ varies,
with $\mathcal G_{(0)}=\{1\}$ and $\mathcal G_{\VR_L}=\mathcal G$.
Following \cite[Section 2.5]{Topics}, a \emph{higher ramification group}
is a nontrivial group in this filtration contained in
$\mathcal G_{\mathcal M_L}$. The adapted-family arguments below use the
full ideal-indexed filtration, including both endpoints.
In the present finite Galois setting, the definition of $I_H$ used above
agrees with the ideal generated by all
$\frac{\sigma b}{b}-1$, $\sigma\in H$, $b\in L^\times$.

We record the closure properties of \cite[Proposition 2.6(1)--(4)]{Topics}
in the form needed below.
\begin{Prop}\label{prop:kuhlmann-duality}
Let $H\leq\mathcal G$ and let $I$ be an $\VR_L$-ideal.
\begin{description}
\item[(i)] $\mathcal G_I$ is the largest subgroup $H'$ such that
$I_{H'}\subseteq I$.
\item[(ii)] $I_H$ is the smallest ideal $I'$ such that
$H\subseteq\mathcal G_{I'}$.
\item[(iii)] We have
\[
I_{\mathcal G_{I_H}}=I_H,
\qquad
\mathcal G_{I_{\mathcal G_I}}=\mathcal G_I.
\]
\end{description}
Consequently, $J\mapsto\mathcal G_J$ and $U\mapsto I_U$ are mutually
inverse inclusion-preserving bijections between the subgroup ideals
$\{I_H\mid H\leq\mathcal G\}$ and the groups in the ideal-indexed
filtration. They restrict to a bijection between the ramification ideals
and the higher ramification groups.
\end{Prop}
\begin{proof}
The definitions give
\[
I_H\subseteq I\quad\Longleftrightarrow\quad H\subseteq\mathcal G_I,
\]
which proves (i) and (ii). Applying this equivalence twice gives
\[
I_H\subseteq I_{\mathcal G_{I_H}}\subseteq I_H,
\qquad
\mathcal G_I\subseteq\mathcal G_{I_{\mathcal G_I}}\subseteq\mathcal G_I.
\]
This proves (iii) and the first bijection. Finally,
$I_H\ne(0)$ if and only if $H\ne\{1\}$, and
$I_H\subseteq\mathcal M_L$ if and only if
$H\subseteq\mathcal G_{\mathcal M_L}$, which gives the restriction.
\end{proof}

For any subgroup $H$, the group $\mathcal G_{I_H}$ is the smallest
member of the ideal-indexed filtration containing $H$, and it has the
same subgroup ideal as $H$. The inclusion $H\subseteq\mathcal G_{I_H}$
can be strict, even when $H$ has prime order. Thus a subgroup of prime
order need not itself be a group in the ramification filtration.

\subsection{Pure extensions and their ramification ideals}

We will use the following consequence of the divisibility of Okutsu degrees.
\begin{Lema}\label{lem:prime-pure}
Let $(F,v)$ be henselian and let $\theta$ have prime degree $p$ over $F$.
Then $\theta$ is pure over $F$. In particular, every finite extension of
henselian valued fields of prime degree is pure.
\end{Lema}
\begin{proof}
By \cite[Corollary 1.18(1)--(2) and Theorem 3.4]{NartOkutsu}, the degrees in
an Okutsu sequence of $\theta$, indexed in increasing order, satisfy
\[
1=m_0<m_1<\cdots<m_r=[F(\theta):F],
\qquad m_i\mid m_{i+1}.
\]
Since $[F(\theta):F]=p$ is prime, the sequence of degrees is $1,p$.
Thus $\theta$ has depth one.
\end{proof}

The following reformulates \cite[Lemma 8.8 and Theorem 8.14]{Dutta}
in terms of distance sets.

\begin{Teo}[Dutta--Novacoski]\label{pureextens}
Let $(L/K,v)$ be finite Galois, with $K$ henselian, and suppose that
$L/K$ is pure on $\theta$.  For every nontrivial subgroup
$H\leq {\rm Gal}(L/K)$, set
\[
\beta_H:=\min_{\sigma\in H\setminus\{1\}}v(\sigma(\theta)-\theta)
\]
and
\begin{equation}\label{defSHtheta}
S_{H,\theta}:=\beta_H-D_1(\theta,K).
\end{equation}
Then
\begin{equation}\label{equationforramific}
I_H=I_{S_{H,\theta}}.
\end{equation}
\end{Teo}

\begin{proof}
If $L/K$ is defectless, choose $a\in K$ such that
$\delta=v(\theta-a)=\max D_1(\theta,K)$.
Since $\theta$ is pure, $(\theta,a)$ is a distinguished pair by
\cite[Remark 2.10(i)]{Dutta}. By
\cite[Lemma 8.8 and its proof]{Dutta}, applied with the translated
generator $\theta-a$,
\[
I_H=\left(\frac{\sigma\theta-\theta}{\theta-a}
\;\middle|\;\sigma\in H\right)=I_{(\beta_H-\delta)^+}.
\]
This is $I_{S_{H,\theta}}$.

If $L/K$ has defect, choose a pseudo-convergent sequence
$(z_\nu)_{\nu<\lambda}$ as in \cite[Remark 2.10]{Dutta}.
By \cite[Theorem 8.14]{Dutta},
\[
I_H=\left(\frac{\sigma\theta-\theta}{\theta-z_\nu}
\;\middle|\;\sigma\in H,\ \nu<\lambda\right).
\]
The values $v(\theta-z_\nu)$ are cofinal in $D_1(\theta,K)$.
Taking valuations of these generators therefore gives
$I_H=I_{\beta_H-D_1(\theta,K)}$.
\end{proof}

We record the individual-automorphism consequence used below.
\begin{Prop}\label{LemmaDutta}
Under the hypotheses of Theorem \ref{pureextens}, for every
$\sigma\in\mathcal G\setminus\{1\}$ we have
\[
I_\sigma=I_{v(\sigma\theta-\theta)-D_1(\theta,K)}.
\]
In particular, for every $b\in L^\times$ with $\sigma b\ne b$,
\begin{equation}\label{eq:DuttaLowerBound}
v\left(\frac{\sigma b-b}{b}\right)
\in v(\sigma\theta-\theta)-D_1(\theta,K).
\end{equation}
\end{Prop}
\begin{proof}
By \cite[Lemma 8.7]{Dutta},
$v(\sigma^j b-b)\geq v(\sigma b-b)$ for every $b\in L$ and $j\geq1$.
Hence $I_\sigma=I_{\langle\sigma\rangle}$ and
$\beta_{\langle\sigma\rangle}=v(\sigma\theta-\theta)$.
Apply Theorem \ref{pureextens} to $\langle\sigma\rangle$.
\end{proof}

\subsection{Principality and defect}

\begin{Cor}\label{princvsdeft}
Let $\mathcal E=(L/K,v)$ be a pure finite Galois extension of henselian
valued fields. Then $\mathcal E$ is defectless if and only if all its
ramification ideals are principal.
\end{Cor}
\begin{proof}
For a pure generator $\theta$, \cite[Remark 2.10]{Dutta} gives two cases.
If $\mathcal E$ is defectless, $D_1(\theta,K)$ has a maximum, so
Theorem \ref{pureextens} makes every subgroup ideal principal.
Otherwise, $\mathcal E$ is immediate and $D_1(\theta,K)$ has no maximum.
For every nontrivial $H$, the final segment
$\beta_H-D_1(\theta,K)$ is then contained in $(vL)_{>0}$ and has no
least element. Thus $I_H$ is a nonprincipal ramification ideal.
\end{proof}

The prime-degree case is an immediate consequence.
\begin{Prop}\label{prop:prime-degree-principality}
Let $(E/F,v)$ be a Galois extension of henselian valued fields of prime degree
$p$. Then the ideal $I_{{\rm Gal}(E/F)}$ is principal if and only if
$(E/F,v)$ is defectless.
\end{Prop}
\begin{proof}
By Lemma \ref{lem:prime-pure}, the extension is pure. Its Galois group
has prime order, so its only nontrivial subgroup ideal is
$I_{{\rm Gal}(E/F)}$. If this ideal is contained in the maximal ideal,
the assertion follows directly from Corollary \ref{princvsdeft}.
Otherwise it is the whole valuation ring and is principal, while there
are no ramification ideals; Corollary \ref{princvsdeft} again implies
that the extension is defectless.
\end{proof}

\section{Products of pure subgroups}\label{sec:productpure}
We now study the case where $\mathcal G=H_1\cdots H_r$ for pure subgroups
$H_i$.
\begin{Def}
A subgroup $H$ of $\mathcal G$ is said to be \textbf{pure} if
$(L/K_H,v)$ is a pure extension, where $K_H=L^H$ is the fixed field of $H$.
\end{Def}

By Lemma \ref{lem:prime-pure}, every subgroup of prime order is pure,
and every generator of its fixed-field extension is a pure generator.

Appendix \ref{app:pure-products} gives examples where $\mathcal G$ is a
product of pure subgroups and where no such product exists. In particular,
this condition is not determined by the abstract group $\mathcal G$.

For the product bound and the question below, we use the following data:
\begin{displaymath}
(*)\qquad
\left\{\begin{array}{l}
\mathcal G=H_1\cdots H_r\mbox{ for pure subgroups }H_i,\ 1\leq i\leq r;\\[6pt]
\mbox{for each }i\mbox{ fix a pure generator }\alpha_i\mbox{ of }
(L/K_{H_i},v);\\[6pt]
\mbox{write }\gamma_{i1},\ldots,\gamma_{is_i}\mbox{ for the distinct values in}\\
\qquad\{v(\sigma\alpha_i-\alpha_i)\mid
\sigma\in H_i\setminus\{1\}\};\\[6pt]
\mathbf D_i=D_1(\alpha_i,K_{H_i}).
\end{array}\right.
\end{displaymath}

\begin{Prop}\label{prop:coordinateideals}
Suppose that we are in situation $(*)$. For each $i$ and $j$ there exists
a subgroup $H\leq H_i$ such that
\[
I_H=I_{\gamma_{ij}-\mathbf D_i}.
\]
\end{Prop}
\begin{proof}
Choose $\sigma\in H_i$ with $v(\sigma\alpha_i-\alpha_i)=\gamma_{ij}$
and put $H=\langle\sigma\rangle$. The inequality
$v(\sigma^k\alpha_i-\alpha_i)\geq v(\sigma\alpha_i-\alpha_i)$
from \cite[Lemma 8.7]{Dutta} shows that the minimum for $H$ is
$\gamma_{ij}$. Apply Theorem \ref{pureextens} over $K_{H_i}$.
\end{proof}

\begin{Def}
Suppose that we are in situation $(*)$ and take a subgroup $H$ of
$\mathcal G$. We say that $\sigma_i\in H_i$ \textbf{appears in $H$} if
there exist $\sigma_j\in H_j$, $j\neq i$, such that
$\sigma_1\cdots\sigma_r\in H$.
\end{Def}

\begin{Teo}\label{mainresultsofar}
Suppose that we are in situation $(*)$. Take a nontrivial subgroup
$H$ of $\mathcal G$ and set
\begin{equation}\label{eqnminH}
\gamma_H:=\min\{v(\sigma_i(\alpha_i)-\alpha_i)-\mathbf D_i\mid
\sigma_i\in H_i\setminus\{1\}\mbox{ appears in }H\}.
\end{equation}
Then
\[
I_H\subseteq I_{\gamma_H}.
\]
Moreover, for every $i$ for which a nontrivial element of $H_i$ appears in
$H$, let $\gamma_i$ be the smallest value
$v(\sigma_i(\alpha_i)-\alpha_i)$ among such elements. If there exists $i$
such that
\begin{equation}\label{casesmallerstrict}
\gamma_i-\mathbf D_i<\gamma_j-\mathbf D_j
\end{equation}
for every $j\neq i$ for which $\gamma_j$ is defined, then
\[
I_H=I_{\gamma_H}.
\]
\end{Teo}
\begin{proof}
Take $\sigma\in H$ and write $\sigma=\sigma_1\cdots\sigma_r$ with
$\sigma_i\in H_i$. Put $\overline\sigma_i=\sigma_{i+1}\cdots\sigma_r$
for $1\leq i\leq r$, with $\overline\sigma_r=1$. Then
\begin{equation}\label{veryimport}
\frac{\sigma b}{b}-1=
\sum_{i=1}^r\frac{\overline\sigma_i b}{b}
\left(\frac{\sigma_i(\overline\sigma_i b)}{\overline\sigma_i b}-1\right).
\end{equation}
Every coefficient $\overline\sigma_i b/b$ is a unit. Proposition
\ref{LemmaDutta}, applied over $K_{H_i}$, places each nonzero summand
in $I_{\gamma_H}$. Hence $I_H\subseteq I_{\gamma_H}$.

Assume now that \eqref{casesmallerstrict} holds. Choose
$\sigma=\sigma_1\cdots\sigma_r\in H$ with the $i$-th component realizing
$\gamma_i$, and put $S_i=\gamma_i-\mathbf D_i$. For every $j\neq i$ for
which $\gamma_j$ is defined, put $S_j=\gamma_j-\mathbf D_j$. By
Proposition \ref{LemmaDutta}, applied over $K_{H_i}$,
\[
I_{\sigma_i}=I_{\langle\sigma_i\rangle}=I_{S_i}.
\]
Since $S_i<S_j$, we have $I_{S_j}\subsetneq I_{S_i}$. Let $J$ be the
union of the finitely many ideals $I_{S_j}$ with $j\neq i$ for which
$\gamma_j$ is defined; if there is no such $j$, put $J=(0)$. Since ideals in
a valuation ring are linearly ordered, $J$ is an ideal and
$J\subsetneq I_{S_i}$.

Let
\[
\mathcal A_i=
\left\{\frac{\sigma_i c}{c}-1\;\middle|\;c\in L^\times\right\}.
\]
By \cite[Lemma 8.7]{Dutta}, $I_{\sigma_i}=I_{\langle\sigma_i\rangle}$,
so $\mathcal A_i$ generates $I_{S_i}$. Since $J\subsetneq I_{S_i}$, some $g_0\in\mathcal A_i$ lies
outside $J$. Moreover, the elements of $\mathcal A_i\setminus J$ already
generate $I_{S_i}$: indeed, if $g\in\mathcal A_i\cap J$, then
$v(g_0)<v(g)$ (because $vJ$ is a final segment and $g_0\notin J$), hence
$g\in(g_0)$.

Now take any
\[
g=\frac{\sigma_i\overline b}{\overline b}-1
\in\mathcal A_i\setminus J
\]
and choose $b$ so that $\overline b=\overline\sigma_i b$. Every other
nonzero summand in \eqref{veryimport} has value in the corresponding final
segment $S_j$, whereas $vg\notin S_j$ for every such $j$. Hence those
summands have strictly larger value than $g$, no cancellation is possible,
and
\[
v\left(\frac{\sigma b}{b}-1\right)=vg.
\]
Thus every element of $\mathcal A_i\setminus J$ generates an ideal contained
in $I_H$. Since these elements generate $I_{S_i}=I_{\gamma_H}$, we obtain
$I_{\gamma_H}\subseteq I_H$. Together with the first inclusion this proves
$I_H=I_{\gamma_H}$.
\end{proof}

Proposition \ref{prop:coordinateideals} identifies the ideals supplied by
each pure factor. It is natural to ask whether these account for all
subgroup ideals of $L/K$.
\begin{Que}\label{ques:pure-factors}
Suppose that $\mathcal G=H_1\cdots H_r$ is a product of pure subgroups.
Does every subgroup ideal of $L/K$ already occur in one of the factors?
More precisely, for every nontrivial $H\leq\mathcal G$, must there exist
an index $i$ and a nontrivial subgroup $Q\leq H_i$ such that
\[
I_H=I_Q?
\]
Equivalently, in the notation $(*)$, must $I_H=I_{\gamma_{ij}-\mathbf D_i}$
for some $i,j$?
\end{Que}
The following examples give a negative answer and show why the
strict-minimum hypothesis in Theorem \ref{mainresultsofar} is needed.

\subsection{A criterion for counterexamples}
\begin{Teo}\label{teo:counterexample-criterion}
Let $(L/K,v)$ be a finite Galois extension of henselian valued fields with
\[
\mathcal G=\langle\sigma\rangle\times\langle\tau\rangle
\simeq C_p\times C_p.
\]
Suppose that $I_\tau\subsetneq I_\sigma$. Then
\[
\mathcal G=\langle\sigma\rangle\times\langle\sigma\tau\rangle,
\qquad
I_{\sigma\tau}=I_\sigma=I_{\mathcal G}.
\]
Consequently, this decomposition into pure subgroups gives a negative
answer to Question \ref{ques:pure-factors}.
\end{Teo}
\begin{proof}
For valuation-preserving automorphisms $\rho,\eta$,
\[
\frac{\rho\eta(b)}b-1
=\frac{\eta(b)}b\left(\frac{\rho(\eta(b))}{\eta(b)}-1\right)
 +\left(\frac{\eta(b)}b-1\right).
\]
Since $\eta(b)/b$ is a unit, $I_{\rho\eta}\subseteq I_\rho+I_\eta$.
Also $I_{\rho^{-1}}=I_\rho$ and, by iteration,
$I_{\langle\rho\rangle}=I_\rho$. Hence
\[
I_{\sigma\tau}\subseteq I_\sigma,
\qquad
I_\sigma\subseteq I_{\sigma\tau}+I_\tau.
\]
Valuation ideals are linearly ordered, so $I_\tau\subsetneq I_\sigma$
forces $I_{\sigma\tau}=I_\sigma$. Every element of $\mathcal G$ is a
product of powers of $\sigma$ and $\tau$, whence $I_{\mathcal G}=I_\sigma$.

Put $H_1=\langle\sigma\rangle$, $H_2=\langle\sigma\tau\rangle$ and
$N=\langle\tau\rangle$. Then $\mathcal G=H_1\times H_2$, with
\begin{equation}\label{strictINIG}
I_N\subsetneq I_{H_1}=I_{H_2}=I_{\mathcal G}.
\end{equation}
Both factors have prime order and are pure by Lemma \ref{lem:prime-pure}.
Proposition \ref{prop:coordinateideals} shows that each supplies only
its own ideal, since it has no proper nontrivial subgroup. Thus neither
supplies $I_N$.
\end{proof}

\begin{Cor}\label{cor:AS-distinct-distances}
Let $\alpha,\beta$ be linearly disjoint Artin--Schreier elements over a
henselian valued field $K$, and set $L=K(\alpha,\beta)$. If the relative
distance cuts of $\alpha$ over $K(\beta)$ and of $\beta$ over $K(\alpha)$
are distinct, then ${\rm Gal}(L/K)$ admits a decomposition into two pure
cyclic factors for which the answer to Question \ref{ques:pure-factors} is negative.
\end{Cor}
\begin{proof}
Let $\sigma$ translate $\alpha$ by $1$ and fix $\beta$, and let $\tau$
translate $\beta$ by $1$ and fix $\alpha$. The extensions
$L/K(\beta)$ and $L/K(\alpha)$ are Galois of degree $p$ and pure on
$\alpha$ and $\beta$, respectively, by Lemma \ref{lem:prime-pure}.
Apply Theorem \ref{pureextens} to these two extensions. Their nontrivial
conjugate differences belong to $\mathbb F_p^\times$, so have value zero;
also $I_{\langle\sigma\rangle}=I_\sigma$ and
$I_{\langle\tau\rangle}=I_\tau$. Thus
\begin{equation}\label{eq:examples-coordinate-ideals}
I_\sigma=I_{-D_1(\alpha,K(\beta))},\qquad
I_\tau=I_{-D_1(\beta,K(\alpha))}.
\end{equation}
These ideals are distinct. After interchanging $\alpha$ and $\beta$ if
necessary, $I_\tau\subsetneq I_\sigma$. Theorem
\ref{teo:counterexample-criterion} shows that
\[
{\rm Gal}(L/K)=\langle\sigma\rangle\times\langle\sigma\tau\rangle
\]
is the required decomposition.
\end{proof}

\begin{Obs}\label{obs:one-ideal-no-counterexample}
The strict inclusion is essential. If all nontrivial subgroup ideals
coincide, every decomposition into order-$p$ factors supplies all
subgroup ideals. In particular, Kuhlmann's degree-$p^2$ example in
\cite[Section 3.5, Proposition 3.22]{Topics} does not satisfy the
strict-inclusion hypothesis of Theorem \ref{intro:counterexample-criterion}
(Theorem \ref{teo:counterexample-criterion}):
it has defect $p$, but
\[
I_H=\left(\frac1\vartheta\right)
\quad\mbox{for every }\{1\}\neq H\leq\mathcal G.
\]
See Subsection \ref{subsec:kuhlmann}. Thus the defect alone does not decide
whether the criterion applies.
\end{Obs}

For the following examples, retain the coordinate automorphisms $\sigma$
and $\tau$ from the proof of Corollary \ref{cor:AS-distinct-distances}.
In the first two examples we use \eqref{eq:examples-coordinate-ideals};
in the last we apply Theorem \ref{pureextens} to a pure generator of $L/K$.
In each example it suffices to verify $I_\tau\subsetneq I_\sigma$;
Theorem \ref{teo:counterexample-criterion} then gives the counterexample
with factors $\langle\sigma\rangle$ and $\langle\sigma\tau\rangle$.

\begin{Exa}[$d=p^2$]\label{exa:counterexample}
In the example from \cite[Section 5.2]{Josnei} with $\theta=\alpha+t\beta$, henselize
the rank-one base field. Then $L=K(\alpha,\beta)$ is immediate of degree
$p^2$, $vt=1$, and
\[
d_1(\alpha,K(\beta))=0^-,\qquad
 d_1(\beta,K(\alpha))=-1^-.
\]
These are the relative cuts computed in \cite[Remark 5.9]{Josnei}.
Here $a^-$ denotes the cut with lower set $\{\gamma:\gamma<a\}$.
Formula \eqref{eq:examples-coordinate-ideals} gives
\[
I_\tau=\{b\in L:vb>1\}\subsetneq\mathcal M_L=I_\sigma.
\]
Rank-one density, applied also to the intermediate fields, preserves
these finite distance cuts under henselization; here
$K(\alpha)=K_0(\alpha)^h$ and $K(\beta)=K_0(\beta)^h$ if $K_0$
denotes the original base.
Moreover, every nonzero $a\alpha+b\beta$, with $a,b\in\mathbb F_p$,
has finite distance cut over the original base ($0^-$ if $b=0$, and
$-1^-$ otherwise), so none enters its henselization. Thus the degree
remains $p^2$. Over the henselian base, the degree-$p$ extensions
$K(\alpha)/K$ and $L/K(\alpha)$ have distance sets without a maximum,
so both are immediate. Hence $L/K$ is immediate.
\end{Exa}

In this example, $L/K_{\langle\sigma\rangle}$ and
$L/K_{\langle\sigma\tau\rangle}$ both have independent defect (AS
distance cut $0^-$), since their subgroup ideals equal $\mathcal M_L$,
but these two pure factors do not supply $I_\tau$.

\begin{Exa}[$d=p$]\label{exa:counterexample-defect-p}
Let $k$ be algebraically closed of characteristic $p$ and put
\[
F=\bigcup_{n\geq0}k((t^{1/p^n})),\qquad wt=1,\qquad
\alpha^p-\alpha=t^{-1},\qquad E=F(\alpha).
\]
The field $F$ is henselian, $wF=\frac{1}{p^\infty}\mathbb Z$ and $Fw=k$,
where $\frac{1}{p^\infty}\mathbb Z:=\bigcup_{n\geq0}\frac{1}{p^n}\mathbb Z$.
For $c_n=\sum_{j=1}^nt^{-1/p^j}$,
\[
w(\alpha-c_n)=-1/p^{n+1}\notin w\bigl(k((t^{1/p^n}))\bigr).
\]
Thus $\alpha\notin F$. Since $wF$ is $p$-divisible and $Fw$ is
algebraically closed, $E/F$ is immediate of degree $p$.
The displayed approximations are cofinal below zero, and
$w(\alpha-c)<0$ for every $c\in F$: a nonnegative value would, by
algebraic closedness of $Fw$ and Hensel's lemma, give an AS root in $F$.
Thus $d_1(\alpha,F)=0^-$.

Choose an irrational real number $\lambda>1$ and an indeterminate $u$.
Extend $w$ to $F(u)$ by
\[
v\left(\sum_j a_ju^j\right)=\min_j\{wa_j+j\lambda\},
\]
and set
\[
K=F(u)^h,\qquad \beta^p-\beta=u^{-1},\qquad
L=K(\alpha,\beta).
\]
All value groups lie in $\mathbb R$, so the valuation has rank one.
We have $v\beta=-\lambda/p$ and $F(u,\beta)=F(\beta)$.
Adjoining a variable whose value is rationally independent of $wE$
preserves the distance cut of $\alpha$ over $F$: for $f,g\in F[X]$,
if $g_jX^j$ is the unique term of least value in $g$, then
\[
v(\alpha-f/g)\leq w(\alpha-f_j/g_j).
\]
Apply this observation with $X=u$ and $X=\beta$. Rank-one density
preserves the finite cut under henselization, giving
\[
d_1(\alpha,K)=d_1(\alpha,K(\beta))=0^-.
\]
Since $v\beta=-\lambda/p\notin vK=\frac{1}{p^\infty}\mathbb Z+\mathbb Z\lambda$,
these cuts imply $[L:K]=p^2$. Using $K(\alpha)=E(u)^h$ and
$L=E(\beta)^h$, the Gauss valuation and immediacy of $E/F$ give
\[
vK(\alpha)=\frac{1}{p^\infty}\mathbb Z+\mathbb Z\lambda,\qquad
vL=\frac{1}{p^\infty}\mathbb Z+\mathbb Z(\lambda/p),\qquad Lv=Kv=k.
\]
Hence $e(L/K)=p$, $f(L/K)=1$ and $d(L/K)=p$.
Since $v\beta\notin vK(\alpha)$, we have
\[
\max D_1(\beta,K(\alpha))=-\lambda/p.
\]
Formula
\eqref{eq:examples-coordinate-ideals} now gives
\[
I_\tau=(1/\beta)=\{b\in L:vb\geq\lambda/p\}
\subsetneq\mathcal M_L=I_\sigma.
\]
\end{Exa}

\begin{Exa}[$d=1$]\label{exa:counterexample-defectless}
Let $K=k((t))$, with $k$ algebraically closed of characteristic $p$ and
$vt=1$, and take
\[
\alpha^p-\alpha=t^{-1},\qquad
\beta^p-\beta=t^{-(p+1)},\qquad L=K(\alpha,\beta).
\]
The Artin--Schreier classes are independent: every nonzero linear
combination has pole order $1$ or $p+1$, neither divisible by $p$.
Thus $[L:K]=p^2$. Set $z=\beta-t^{-1}\alpha$. Then
\[
z^p-z=(t^{-1}-t^{-p})\alpha,\qquad
v\alpha=-1/p,\qquad vz=-1-1/p^2.
\]
The denominator of $vz$ forces $e(L/K)=p^2$, hence $d(L/K)=1$.
It also shows that $K(z)=L$ and that $z$ is pure: if
$[K(c):K]<p^2$, then $vc\ne vz$, since otherwise
$e(K(c)/K)\geq p^2$. Consequently $v(z-c)\leq vz=v(z-0)$,
so $(z,0)$ is a distinguished pair and $\max D_1(z,K)=vz$.

Here we apply Theorem \ref{pureextens} directly to $z$ over $K$.
For $1\leq j<p$, the differences $\sigma^jz-z=-jt^{-1}$ and
$\tau^jz-z=j$ have values $-1$ and $0$. Therefore
\[
I_\tau=\{b\in L:vb\geq1+1/p^2\}
\subsetneq\{b\in L:vb\geq1/p^2\}=I_\sigma.
\]
\end{Exa}

\begin{Cor}\label{cor:all-defects-counterexamples}
For every prime $p$ and every $d\in\{1,p,p^2\}$, there is a henselian
rank-one valued field of characteristic $p$ and a Galois extension $L/K$ with
$\mathcal G\simeq C_p^2$ and defect $d$, admitting a decomposition into
two pure subgroups for which the answer to Question \ref{ques:pure-factors} is negative.
\end{Cor}
\begin{proof}
Use Examples \ref{exa:counterexample-defectless},
\ref{exa:counterexample-defect-p} and \ref{exa:counterexample}.
\end{proof}

\begin{Obs}\label{rmk:cancellation}
Theorem \ref{teo:counterexample-criterion} explains the obstruction in
Theorem \ref{mainresultsofar}. With its decomposition
$\mathcal G=H_1\times H_2$,
\[
\tau=\sigma^{-1}(\sigma\tau).
\]
The two components have the same first ramification level, but their product
lies in the deeper group $N$. Thus the terms of minimum value in
\eqref{veryimport} can cancel for structural ramification-theoretic reasons.
The strict-minimum hypothesis in Theorem \ref{mainresultsofar} is therefore
not merely an artifact of the proof.
\end{Obs}

\subsection{Ramification-adapted families and decompositions}
We now require the chosen subgroups to generate each ramification group
through their intersections with it. For a direct product of order-$p$
subgroups, this means that every ramification group is a product of some
of the chosen factors.

\begin{Def}\label{def:adapted}
Let $H_1,\ldots,H_r$ be nontrivial subgroups of $\mathcal G$.  We call the
family $(H_1,\ldots,H_r)$ \textbf{ramification-adapted} if every ideal-indexed
ramification group $U=\mathcal G_I$ satisfies
\[
U=\left\langle U\cap H_1,\ldots,U\cap H_r\right\rangle.
\]
Here $U$ runs through the ramification filtration, and the brackets mean
the subgroup generated by the indicated intersections. A product
decomposition is adapted when its factors form a ramification-adapted
family. For an internal direct product, the condition is equivalent
to $U=\prod_i(U\cap H_i)$: every component of every element of $U$
belongs to $U$. If, moreover,
\[
\mathcal G=H_1\times\cdots\times H_r
\]
is an internal direct product and every $U=\mathcal G_I$ is of the form
\[
U=\prod_{i\in A}H_i
\]
for some $A\subseteq\{1,\ldots,r\}$, we call the decomposition
\textbf{block-ramification-adapted}. This stronger condition requires
each ramification group to be a product of entire factors.

Every block-ramification-adapted decomposition is ramification-adapted.  If
$\mathcal G\simeq C_p^r$ and the $H_i$ are the one-dimensional factors of a
direct product decomposition, the two notions coincide, because
$U\cap H_i$ is either $\{1\}$ or $H_i$.
\end{Def}

\begin{Cor}\label{cor:disjoint-factor-ideals}
Assume $(*)$ and put
\[
\mathcal S_i=\{I_{\gamma_{ij}-\mathbf D_i}\mid 1\leq j\leq s_i\}.
\]
If the sets $\mathcal S_i$ are pairwise disjoint, then
$(H_1,\ldots,H_r)$ is ramification-adapted and
\[
\mathcal I(\mathcal E)=\bigcup_{i=1}^r\mathcal S_i.
\]
\end{Cor}
\begin{proof}
For every nontrivial subgroup $H$, pairwise disjointness makes the
minimum among the factor final segments in Theorem \ref{mainresultsofar}
unique. Thus $I_H=I_{\gamma_H}$, and the displayed equality follows
from Proposition \ref{prop:coordinateideals}.
Now take $U=\mathcal G_I\neq\{1\}$ and write
$\sigma=\sigma_1\cdots\sigma_r\in U$, with $\sigma_i\in H_i$.
Each component appears in $U$, so the same equality gives
$I_{\sigma_i}\subseteq I_U\subseteq I$. Hence every $\sigma_i$ belongs
to $U$, proving $U=\prod_i(U\cap H_i)$. The case $U=\{1\}$ is immediate.
\end{proof}

\begin{Obs}\label{obs:adaptation-no-cancellation}
For an internal direct product $\mathcal G=H_1\times\cdots\times H_r$,
ramification adaptation is equivalent to
\[
I_{\sigma_1\cdots\sigma_r}=\sum_{i=1}^r I_{\sigma_i}
\qquad(\sigma_i\in H_i).
\]
Indeed, adaptation says that $\sigma_1\cdots\sigma_r\in\mathcal G_I$
if and only if every $\sigma_i\in\mathcal G_I$; use
$\rho\in\mathcal G_I\Longleftrightarrow I_\rho\subseteq I$.
Since the sum is the largest component ideal, this expresses the absence
of cancellation between component ideals.
\end{Obs}

\begin{Teo}\label{teo:adapted}
Let $H_1,\ldots,H_r$ be a ramification-adapted family.  Then, for every
nontrivial subgroup $H\leq\mathcal G$, there exist an index $i$ and a
nontrivial cyclic subgroup $Q\leq H_i$ such that
\[
I_H=I_Q.
\]
If, in addition, every $H_i$ is pure, fix a pure generator $\alpha_i$ of
$(L/K_{H_i},v)$ and put
\[
\mathbf D_i=D_1(\alpha_i,K_{H_i}).
\]
Write $\gamma_{i1},\ldots,\gamma_{is_i}$ for the distinct values in
\[
\{v(\rho\alpha_i-\alpha_i)\mid \rho\in H_i\setminus\{1\}\}.
\]
Then $Q$ can be chosen so that
\[
I_H=I_Q=I_{\beta_{Q,i}-\mathbf D_i},
\qquad
\beta_{Q,i}:=
\min_{\rho\in Q\setminus\{1\}}v(\rho\alpha_i-\alpha_i).
\]
In particular, $\beta_{Q,i}$ is one of the values $\gamma_{ij}$, and hence
\[
I_H=I_{\gamma_{ij}-\mathbf D_i}
\]
for suitable $i,j$.

If every $H_i$ has order $p$, all the factors are pure by Lemma
\ref{lem:prime-pure}, and necessarily $Q=H_i$. Thus the distance-set
formula above applies to any choice of generators $\alpha_i$, and
\[
I_H=I_{H_i}.
\]
In particular, this applies to adapted decompositions of elementary
abelian groups.
\end{Teo}
\begin{proof}
Set
\[
U:=\mathcal G_{I_H}.
\]
By Proposition \ref{prop:kuhlmann-duality}(i)--(iii), $H\subseteq U$,
$I_U=I_H$, and $U$ is the smallest ideal-indexed ramification group
containing $H$: indeed, if $H\subseteq\mathcal G_J$, then
$I_H\subseteq J$, hence
$U=\mathcal G_{I_H}\subseteq\mathcal G_J$.

Because the chain of ideal-indexed ramification groups is finite,
there is a largest member $V$ properly contained in $U$; if $U$ is the
smallest nontrivial member of the chain, take $V=\{1\}$.  Since the family is
ramification-adapted, not all intersections $U\cap H_i$ can be contained in
$V$.  Choose
\[
\rho\in (U\cap H_i)\setminus V
\]
for some $i$, and set $Q=\langle\rho\rangle$.  Then $Q\leq U$ and
$Q\nsubseteq V$.  If $W$ were an ideal-indexed ramification group with
$Q\subseteq W\subsetneq U$, then $W\subseteq V$, because $V$ is the largest
proper member below $U$, contradicting $Q\nsubseteq V$.  Consequently $U$ is
the smallest ideal-indexed ramification group containing $Q$.
By Proposition \ref{prop:kuhlmann-duality} again,
\[
\mathcal G_{I_Q}=U
\qquad\text{and}\qquad
I_Q=I_U=I_H.
\]
For the second assertion, assume in addition that the $H_i$ are pure.
Theorem \ref{pureextens}, applied to the pure extension $L/K_{H_i}$ and its
subgroup $Q$, gives
\[
I_Q=
I_{\min_{\tau\in Q\setminus\{1\}}
\{v(\tau\alpha_i-\alpha_i)-\mathbf D_i\}}
=I_{\beta_{Q,i}-\mathbf D_i}.
\]
The minimum is attained by some element of $Q\setminus\{1\}\subseteq H_i$,
so $\beta_{Q,i}=\gamma_{ij}$ for a suitable $j$.  Finally, if $|H_i|=p$,
the nontrivial subgroup $Q\leq H_i$ must be $H_i$ itself.
\end{proof}

\begin{Cor}\label{cor:block-adapted}
Suppose that
\[
\mathcal G=H_1\times\cdots\times H_r
\]
is block-ramification-adapted. Then, for every nontrivial subgroup
$H\leq\mathcal G$, there exists $i$ such that
\[
I_H=I_{H_i}.
\]
If, in addition, every $H_i$ is pure and the data
$\alpha_i,\gamma_{ij},\mathbf D_i$ are fixed as in $(*)$, then there exists
$j$ such that
\[
I_H=I_{H_i}=I_{\gamma_{ij}-\mathbf D_i}.
\]
\end{Cor}
\begin{proof}
Let $U=\mathcal G_{I_H}$ and let $V$ be its largest proper predecessor, with
$V=\{1\}$ if $U$ is the smallest nontrivial member.  Block adaptation gives
subsets $B\subsetneq A$ such that
\[
U=\prod_{i\in A}H_i,
\qquad
V=\prod_{i\in B}H_i.
\]
For $i\in A\setminus B$, we have $H_i\leq U$ and $H_i\nsubseteq V$.
Therefore $U$ is the smallest ideal-indexed ramification group
containing $H_i$, and Proposition \ref{prop:kuhlmann-duality} yields
$I_H=I_U=I_{H_i}$.  If $H_i$ is pure, Theorem \ref{pureextens} gives
\[
I_{H_i}=I_{\beta_i-\mathbf D_i},
\qquad
\beta_i=\min_{\rho\in H_i\setminus\{1\}}
v(\rho\alpha_i-\alpha_i).
\]
Since $\beta_i=\gamma_{ij}$ for some $j$, the final assertion follows.
\end{proof}

\begin{Cor}\label{cor:pure-adapted-principality}
Let $H_1,\ldots,H_r$ be a ramification-adapted family of pure subgroups.
Then the following conditions are equivalent:
\begin{enumerate}
\item every ramification ideal of $\mathcal E=(L/K,v)$ is principal;
\item the pure extension $(L/K_{H_i},v)$ is defectless for every $i$.
\end{enumerate}
\end{Cor}
\begin{proof}
Assume first that every ramification ideal of $\mathcal E$ is principal.
Fix $i$ and a nontrivial subgroup $Q\leq H_i$.  If $I_Q$ is contained in
$\mathcal M_L$, it is a ramification ideal of $\mathcal E$ and is principal
by assumption.  Otherwise $I_Q=\VR_L$, which is also principal.  Thus every
ramification ideal of the pure extension $L/K_{H_i}$ is principal, and
Corollary \ref{princvsdeft} shows that $L/K_{H_i}$ is defectless.

Conversely, suppose that every $L/K_{H_i}$ is defectless.  Given a
nontrivial subgroup $H\leq\mathcal G$, Theorem \ref{teo:adapted} gives an
index $i$ and a subgroup $Q\leq H_i$ such that $I_H=I_Q$.  Corollary
\ref{princvsdeft}, applied to the pure defectless extension $L/K_{H_i}$,
shows that $I_Q$ is principal whenever it is a ramification ideal; otherwise
$I_Q=\VR_L$.  Hence every ramification ideal of $\mathcal E$ is principal.
\end{proof}

\begin{Cor}\label{cor:adaptedprincipality}
Let $(L/K,v)$ be a finite Galois extension of henselian valued fields with
\[
\mathcal G\simeq C_p^r,
\]
and let
\[
\mathcal G=H_1\times\cdots\times H_r,
\qquad |H_i|=p,
\]
be a ramification-adapted decomposition. Then the following conditions are
equivalent:
\begin{enumerate}
\item every ramification ideal of $\mathcal E=(L/K,v)$ is principal;
\item $I_{H_i}$ is principal for every $i$;
\item the degree-$p$ extension $(L/K_{H_i},v)$ is defectless for every $i$.
\end{enumerate}
These equivalent conditions do not, in general, imply that $(L/K,v)$ is
defectless.
\end{Cor}
\begin{proof}
By Theorem \ref{teo:adapted}, every nontrivial subgroup ideal equals
some $I_{H_i}$. Each such ideal is either a ramification ideal or
$\VR_L$, so (1) and (2) are equivalent. Proposition
\ref{prop:prime-degree-principality} gives (2)$\Leftrightarrow$(3).
Kuhlmann's example in Subsection \ref{subsec:kuhlmann} proves the final
assertion.
\end{proof}

\begin{Cor}\label{cor:adaptedexists}
Let $(L/K,v)$ be a finite Galois extension of henselian valued fields with
\[
\mathcal G\simeq C_p^r.
\]
Then there exists a ramification-adapted decomposition into pure subgroups
\[
\mathcal G=H_1\times\cdots\times H_r,
\qquad |H_i|=p.
\]
Consequently, for every nontrivial $H$, there is an index $i$ with
$I_H=I_{H_i}$. For any generator $\alpha_i$ of $L/K_{H_i}$, this gives
\[
I_H=I_{\beta_i-D_1(\alpha_i,K_{H_i})},
\qquad
\beta_i=\min_{\rho\in H_i\setminus\{1\}}v(\rho\alpha_i-\alpha_i).
\]
\end{Cor}
\begin{proof}
Regard $\mathcal G$ as an $r$-dimensional vector space over $\mathbb F_p$.
The distinct ideal-indexed ramification groups form a finite flag of subspaces
\[
\mathcal G=U_0\supsetneq U_1\supsetneq\cdots\supsetneq U_s=\{0\}.
\]
Choose a basis of $\mathcal G$ adapted to this flag, i.e. a basis such that
each $U_j$ is spanned by a subfamily of the basis. If $H_i$ is the line
spanned by the $i$-th basis vector, the resulting direct product is
block-ramification-adapted, and hence ramification-adapted. Each $H_i$ has
order $p$, so every generator of $L/K_{H_i}$ is pure by Lemma
\ref{lem:prime-pure}. The remaining assertions follow from Theorems
\ref{teo:adapted} and \ref{pureextens}.
\end{proof}

The value $\beta_i$ depends on the generator. If $F=K_{H_i}$ and
$\alpha_i$ is replaced by $a\alpha_i+c$, with $a\in F^\times$ and $c\in F$,
then $\beta_i$ and $D_1(\alpha_i,F)$ are both translated by $va$. Hence the
final segment $\beta_i-D_1(\alpha_i,F)$, and therefore its ideal, is
unchanged. For an Artin--Schreier generator, $\beta_i=0$.

Combining Corollaries \ref{cor:adaptedprincipality} and
\ref{cor:adaptedexists} gives a criterion stated entirely in terms of defect.
\begin{Cor}\label{cor:top-degree-p-principality}
Let $(L/K,v)$ be a finite Galois extension of henselian valued fields with
$\mathcal G\simeq C_p^r$. Then all ramification ideals of $L/K$ are
principal if and only if $(L/F,v)$ is defectless for every intermediate
field $F$ with $[L:F]=p$.
\end{Cor}
\begin{proof}
Suppose first that all ramification ideals of $L/K$ are principal.
For an intermediate field $F$ with $[L:F]=p$, put
$H={\rm Gal}(L/F)$. The ideal $I_H$ is either a ramification ideal of
$L/K$ or the whole valuation ring, so it is principal. Proposition
\ref{prop:prime-degree-principality} implies that $(L/F,v)$ is defectless.

Conversely, suppose every such $(L/F,v)$ is defectless. By Corollary
\ref{cor:adaptedexists}, choose a ramification-adapted decomposition
\[
\mathcal G=H_1\times\cdots\times H_r,\qquad |H_i|=p.
\]
Each extension $L/K_{H_i}$ has degree $p$ and is therefore defectless by
assumption. Corollary \ref{cor:adaptedprincipality} shows that every
ramification ideal of $L/K$ is principal.
\end{proof}

Kuhlmann's example in Subsection \ref{subsec:kuhlmann} satisfies both
equivalent conditions of Corollary \ref{cor:top-degree-p-principality}:
all its ramification ideals are principal, and every top degree-$p$
extension $L/F$ is defectless. Yet $d(L/K,v)=p$. Thus the corollary
characterizes defectlessness of all these top extensions, which is
strictly weaker than defectlessness of $L/K$.

\begin{Cor}\label{Corkuhlmn}
Assume $(*)$ and let $H\leq\mathcal G$ be nontrivial. Suppose an index
$i$ attaining the minimum in \eqref{eqnminH} satisfies
\[
\max\mathbf D_i=v\alpha_i,\qquad
\rho(\tau\alpha_i)=\tau\alpha_i
\quad(\tau\in H_i,\ \rho\in H_j,\ j\neq i).
\]
Then $I_H=I_{\gamma_H}$.
\end{Cor}
\begin{proof}
Choose $\sigma=\sigma_1\cdots\sigma_r\in H$ whose $i$-th component
attains the minimum. The final segment $\gamma_H$ has least element
\[
\delta=v(\sigma_i\alpha_i-\alpha_i)-v\alpha_i.
\]
The invariance hypothesis gives $\sigma\alpha_i=\sigma_i\alpha_i$.
Thus $v(\sigma\alpha_i/\alpha_i-1)=\delta$, so $I_H$ contains
$I_{\gamma_H}$. The reverse inclusion is Theorem \ref{mainresultsofar}.
\end{proof}

\section{Applications to Artin--Schreier extensions}\label{sec:ASapplications}

We now apply the preceding results to elementary abelian extensions obtained
as composita of Artin--Schreier extensions.  In this setting the canonical
cyclic factors have order $p$ and are therefore pure. Their ramification
ideals admit a particularly simple description.  We then
return to Kuhlmann's degree-$p^2$ example and explain its relation with the
ramification-adapted theory developed above.

\subsection{Linearly disjoint Artin--Schreier composita}

Let $\operatorname{char}(K)=p>0$.  Take Artin--Schreier elements
\[
\alpha_1,\ldots,\alpha_n\in\overline K
\]
such that the extensions $K(\alpha_1),\ldots,K(\alpha_n)$ are linearly
disjoint over $K$, and set
\[
L=K(\alpha_1,\ldots,\alpha_n).
\]
For each $i$ put
\[
K_i=K(\alpha_1,\ldots,\alpha_{i-1},\alpha_{i+1},\ldots,\alpha_n),
\qquad
H_i={\rm Gal}(L/K_i).
\]
Then $|H_i|=p$ and $L=K_i(\alpha_i)$.

\begin{Prop}\label{prop:AS-coordinate-ideal}
Assume that $(K,v)$ is henselian.  With the notation above, set
\[
\mathbf D_i=D_1(\alpha_i,K_i).
\]
Then, for every $i$,
\[
I_{H_i}=I_{-\mathbf D_i}.
\]
The formula holds in both the defectless and the defect cases.
\end{Prop}

\begin{proof}
The field $K_i$ is henselian and $[L:K_i]=p$, so $\alpha_i$ is pure over
$K_i$ by Lemma \ref{lem:prime-pure}. Every nontrivial $\sigma\in H_i$
satisfies $\sigma\alpha_i-\alpha_i\in\mathbb F_p^\times$, of value zero.
Theorem \ref{pureextens}, applied over $K_i$, gives the formula.
\end{proof}

For the next result, say that $D_1(\alpha,K)$ is \emph{cofinal below $0$}
if for every $\delta\in vK$ with $\delta<0$ there is $a\in K$ such that
$\delta<v(\alpha-a)<0$. For an immediate AS extension in rank one,
this is equivalent to $d_1(\alpha,K)=0^-$.

\begin{Cor}\label{linealrydisoj}
Let $(K,v)$ be henselian and let
$L=K(\alpha_1,\ldots,\alpha_n)$ be an immediate compositum of linearly
disjoint Artin--Schreier extensions. If each $D_1(\alpha_i,K)$ is
cofinal below $0$, then
\[
I_\sigma=\mathcal M_L\quad(1\neq\sigma\in{\rm Gal}(L/K)),
\qquad {\rm Ram}(L/K,v)=\{\mathcal M_L\}.
\]
\end{Cor}
\begin{proof}
Immediacy gives $I_\sigma\subseteq\mathcal M_L$: for $b\in L^\times$,
choose $c\in K$ with $v(b-c)>vb$. Then $v\sigma(b-c)=v(b-c)$, so
$v(\sigma b-b)\geq v(b-c)>vb$, hence $v(\sigma b/b-1)>0$ whenever $\sigma b\neq b$.

Fix $\sigma\neq1$ and choose $i$ with
$\sigma\alpha_i-\alpha_i\in\mathbb F_p^\times$. For $a\in K$,
\[
v\left(\frac{\sigma(\alpha_i-a)}{\alpha_i-a}-1\right)
=-v(\alpha_i-a).
\]
These values are coinitial in $(vL)_{>0}$ by the assumed cofinality and
$vL=vK$. Thus $\mathcal M_L\subseteq I_\sigma$.
\end{proof}

\begin{Obs}\label{obs:independent-defect-recovery}
This recovers the ramification-ideal conclusion of
\cite[Corollary 6.10]{Josnei} for the immediate composita considered there.
The argument does not require additional assumptions on constants in $K$.
\end{Obs}

\subsection{Kuhlmann's degree-$p^2$ example}\label{subsec:kuhlmann}

In \cite[Section 3.5]{Topics}, Kuhlmann constructs a tower
\[
K\subset L_0=K(\vartheta_0)\subset L=L_0(\vartheta),
\qquad
\vartheta_0^p-\vartheta_0=s,\quad
\vartheta^p-\vartheta=\vartheta_0,
\]
whose lower step is immediate of degree $p$, while the upper step has
ramification index $p$. Thus $[L:K]=p^2$ and $d(L/K)=p$.
The base contains $\zeta$ with $\zeta^p-\zeta=1$. The commuting
order-$p$ automorphisms
\[
\sigma(\vartheta_0)=\vartheta_0+1,\quad
\sigma(\vartheta)=\vartheta+\zeta,\qquad
\tau(\vartheta_0)=\vartheta_0,\quad
\tau(\vartheta)=\vartheta+1
\]
generate $\mathcal G\simeq C_p^2$.

The calculation in \cite[Section 3.5, proof of Proposition 3.22]{Topics}
gives
\[
I_H=(1/\vartheta)
\qquad(\{1\}\neq H\leq\mathcal G),
\qquad v(1/\vartheta)>0.
\]
Unlike Example \ref{exa:counterexample-defect-p}, this extension has only
one nontrivial subgroup ideal, despite having the same defect $p$.
In particular, for every pair of generators $\rho,\eta$ of $\mathcal G$,
\[
I_\rho=I_\eta=I_{\rho\eta}=I_{\mathcal G}=(1/\vartheta).
\]
Thus Kuhlmann's example does not satisfy the hypotheses of
Theorem \ref{intro:counterexample-criterion}
(Theorem \ref{teo:counterexample-criterion}): the required strict
inclusion $I_\eta\subsetneq I_\rho$ fails for every choice of generators.
Indeed, every decomposition into two order-$p$ factors supplies all
nontrivial subgroup ideals, so this example gives an affirmative answer
to Question \ref{ques:pure-factors}.

By Proposition \ref{prop:kuhlmann-duality}, the only nontrivial
ideal-indexed ramification group is $\mathcal G$. Every basis is therefore
ramification-adapted. More directly, Corollary
\ref{cor:top-degree-p-principality} identifies the principality of all
ramification ideals here with defectlessness of every top degree-$p$
extension. Thus, for every intermediate field $F$ with $[L:F]=p$,
\[
d(L/F,v)=1,\qquad d(L/K,v)=p.
\]
By multiplicativity of the defect, $d(F/K,v)=p$: in every such tower
$K\subset F\subset L$, the defect occurs in the lower step.
This is the situation highlighted in Theorem \ref{intro:principality}
and the discussion following it. In particular, neither principality
of all ramification ideals nor defectlessness of every top degree-$p$
extension implies defectlessness of $L/K$.

\appendix
\section{A recursive Galois criterion for Artin--Schreier towers}
\label{app:AS-galois}
Throughout this appendix, $\operatorname{char}(K)=p>0$ and
$\wp(x)=x^p-x$. The following extends \cite[Lemma 2.9]{Nart} from a
Galois extension of degree $p$ to a finite Galois $p$-extension.
No valuation hypothesis is required.

\begin{Teo}\label{teo:AS-one-step}
Let $L/K$ be a finite Galois $p$-extension with group $\mathcal G$, and let
$M=L(\eta)$, where $\wp(\eta)=b\in L\setminus\wp(L)$.
Then $M/K$ is Galois if and only if
\begin{equation}\label{eq:AS-lifting}
\sigma(b)-b\in\wp(L)\qquad\mbox{for every }\sigma\in \mathcal G.
\end{equation}
It suffices to check this on a generating set of $\mathcal G$. Whenever
$\wp(c_\sigma)=\sigma(b)-b$ with $c_\sigma\in L$, a lift of $\sigma$
is given by
\[
\widetilde\sigma|_L=\sigma,\qquad
\widetilde\sigma(\eta)=\eta+c_\sigma.
\]
Together with $\tau|_L=1$, $\tau(\eta)=\eta+1$, lifts of generators
of $\mathcal G$ generate ${\rm Gal}(M/K)$.
\end{Teo}
\begin{proof}
Suppose $M/K$ is Galois. Its group is a $p$-group, so the normal subgroup
${\rm Gal}(M/L)=\langle\tau\rangle\simeq C_p$ is central:
its automorphism group has order $p-1$. For any lift
$\widetilde\sigma$ of $\sigma$, centrality implies that
$c_\sigma=\widetilde\sigma(\eta)-\eta$ is fixed by $\tau$, hence lies
in $L$. Applying $\wp$ gives \eqref{eq:AS-lifting}.

Conversely, a solution $c_\sigma\in L$ gives the displayed lift because
$\wp(\eta+c_\sigma)=\sigma(b)$. Every $K$-embedding of $M$ restricts
to some $\sigma\in \mathcal G$ and sends $\eta$ to
$\eta+c_\sigma+j$ for some $j\in\mathbb F_p$. Thus its image is $M$,
proving normality; separability is automatic.
Finally, the elements satisfying \eqref{eq:AS-lifting} form the
stabilizer of $[b]$ in $L/\wp(L)$, so checking generators suffices.
Lifts of those generators, together with the kernel $\langle\tau\rangle$,
generate the full group.
\end{proof}

\begin{Cor}\label{teo:AS-tower-recursive}
Let $K=K_0\subset K_1\subset\cdots\subset K_n$, where
$K_i=K_{i-1}(\alpha_i)$ and
$\wp(\alpha_i)=a_i\in K_{i-1}\setminus\wp(K_{i-1})$.
Starting with $K_1/K$, apply Theorem \ref{teo:AS-one-step} successively:
once $K_{i-1}/K$ is Galois, $K_i/K$ is Galois precisely when
\[
\sigma(a_i)-a_i\in\wp(K_{i-1})
\]
for every $\sigma$ in a generating set of ${\rm Gal}(K_{i-1}/K)$.
At a successful step, the lifts and the new translation give generators
for the next test.
\end{Cor}
This tests whether every stage is Galois over $K$; the top extension
$K_n/K$ alone can be Galois even when some intermediate stage is not.

\subsection*{Examples from \cite{Dutta,Nart}}
We apply the criterion to the towers in \cite[Example 9.1]{Dutta}
and \cite[Sections 4.2--4.4]{Nart}, using $t$ for the variable denoted
by $y$ in the latter paper.

\begin{Exa}[An elementary abelian extension]\label{exa:AS-split}
Take the extension of \cite[Example 9.1]{Dutta}. Let $k$ be an
Artin--Schreier closed field of characteristic $p>2$, let
$K=k(t)^h$ for the $t$-adic valuation, and set
\[
\wp(\alpha)=t^{-1},\qquad \wp(\beta)=\alpha+t^{-1},\qquad
L=K(\alpha),\qquad M=L(\beta).
\]
We have $v\alpha=-1/p$ and $v(\beta-\alpha)=-1/p^2$, so both
steps have degree $p$. Choose $c\in k$ with $\wp(c)=1$.
For $\sigma(\alpha)=\alpha+1$,
\[
\sigma(\alpha+t^{-1})-(\alpha+t^{-1})=1=\wp(c).
\]
Thus $\sigma$ lifts by $\widetilde\sigma(\beta)=\beta+c$.
This lift and the translation fixing $L$ and sending $\beta$ to
$\beta+1$ commute and have order $p$. Hence
${\rm Gal}(M/K)\simeq C_p\times C_p$.
\end{Exa}

For the next examples, use the base field from
\cite[Section 3.2]{Nart}:
\[
K=\bigcup_{n\geq0}k((t^{1/p^n})),\qquad
k=\overline{\mathbb F_p},\qquad vt=1.
\]
In the Hahn field $k((t^{\mathbb Q}))$, take
\[
\alpha=\sum_{n\geq1}t^{-1/p^n},\qquad
\wp(\alpha)=t^{-1},\qquad L=K(\alpha).
\]
This is the immediate degree-$p$ extension of
\cite[Section 4.1]{Nart}. Fix $c\in k$ with $\wp(c)=1$.

\begin{Exa}[The cyclic quartic defect tower]\label{exa:AS-cyclic-four}
Let $p=2$ and take the tower of \cite[Section 4.2]{Nart}:
\[
M=L(\eta),\qquad \eta^2-\eta=t^{-1}\alpha.
\]
Both steps have degree $2$. For $\sigma(\alpha)=\alpha+1$,
\[
\sigma(t^{-1}\alpha)-t^{-1}\alpha=t^{-1}=\wp(\alpha).
\]
The criterion gives a lift $\widetilde\sigma(\eta)=\eta+\alpha$.
Since $\widetilde\sigma^2(\eta)=\eta+1$, the lift has order $4$ and
${\rm Gal}(M/K)\simeq C_4$. The extension is immediate and has depth
$2$, as proved in \cite[Section 4.2]{Nart}; we use this fact in
Example \ref{exa:no-pure-product}.
\end{Exa}

\begin{Exa}[The non-Galois tower]\label{exa:AS-not-Galois}
Let $p>2$ and take the tower of \cite[Section 4.3]{Nart}:
\[
M=L(\eta),\qquad \wp(\eta)=\alpha^2.
\]
The classes of $\alpha$ and $\alpha^2$ in $L/\wp(L)$ are linearly
independent, as shown there. In particular, $[M:L]=p$.
For $\sigma(\alpha)=\alpha+1$, the test gives
\[
\sigma(\alpha^2)-\alpha^2=2\alpha+1\notin\wp(L),
\]
because $1=\wp(c)$ and $[\alpha]\ne0$ in $L/\wp(L)$.
Hence $M/K$ is not Galois.
\end{Exa}

\begin{Exa}[The nonabelian normal closure]\label{exa:AS-degree-p3-test}
Continue the preceding example and adjoin
\[
\wp(\theta)=\alpha,\qquad K_2=L(\theta),\qquad
N=K_2(\eta).
\]
The independence of $[\alpha]$ and $[\alpha^2]$ gives $[N:K]=p^3$.
The first test is $\sigma(\alpha)-\alpha=1=\wp(c)$, so $K_2/K$
is Galois. Its group is generated by
\[
\widetilde\sigma(\alpha)=\alpha+1,\quad
\widetilde\sigma(\theta)=\theta+c,\qquad
\tau(\alpha)=\alpha,\quad \tau(\theta)=\theta+1.
\]
For the last step, the two tests are
\[
\widetilde\sigma(\alpha^2)-\alpha^2
=2\alpha+1=\wp(2\theta+c),\qquad
\tau(\alpha^2)-\alpha^2=0.
\]
Thus $N/K$ is Galois. The lifts, denoted by the same letters, satisfy
\[
\widetilde\sigma(\eta)=\eta+2\theta+c,\qquad \tau(\eta)=\eta,
\]
and do not commute: their two compositions on $\eta$ differ by $2$.
Together with $\iota(\eta)=\eta+1$, fixing $K_2$, they give
\[
{\rm Gal}(N/K)\simeq(C_p\times C_p)\rtimes C_p.
\]
Here $\langle\tau,\iota\rangle\simeq C_p^2$ is normal and
$\widetilde\sigma$ has order $p$. This is the normal closure from
\cite[Section 4.3]{Nart}: any normal extension containing $M$ must also
contain $\widetilde\sigma(\eta)-\eta=2\theta+c$, hence $\theta$.
Its depth is $2$ by \cite[Section 4.4]{Nart}.
\end{Exa}

\section{Examples of products of pure subgroups}\label{app:pure-products}
We illustrate the hypothesis $\mathcal G=H_1\cdots H_r$ with each
$(L/K_{H_i},v)$ pure. As in Section \ref{sec:productpure}, the factors need
not be proper: if $(L/K,v)$ is pure, the single factor $H_1=\mathcal G$
satisfies the condition.

\begin{Exa}[Elementary abelian extensions]\label{exa:pure-product-elementary}
If $\mathcal G\simeq C_p^r$, choose a basis and let $H_1,\ldots,H_r$
be the corresponding order-$p$ subgroups. Then
\[
\mathcal G=H_1\times\cdots\times H_r,
\]
and every $H_i$ is pure by Lemma \ref{lem:prime-pure}.
This applies to Example \ref{exa:AS-split} from
\cite[Example 9.1]{Dutta}, and to the example from \cite[Section 5.2]{Josnei}
used in Example \ref{exa:counterexample}. In particular,
the variation of Abhyankar's extension in
Example \ref{exa:counterexample-defect-p} satisfies it although $(L/K,v)$
is not pure: it has nontrivial defect and is not immediate, whereas
a pure defect extension is immediate \cite[Remark 2.10(ii)]{Dutta}.
\end{Exa}

For a cyclic group $\mathcal G\simeq C_{p^n}$, every proper subgroup
is contained in its unique maximal subgroup. Hence a product of
subgroups equals $\mathcal G$ only if one factor is $\mathcal G$ itself.
Consequently, in this case the product hypothesis holds if and only if
$(L/K,v)$ is pure. The following two rank-one examples illustrate both possibilities
for $\mathcal G\simeq C_4$.

\begin{Exa}[A cyclic quartic extension satisfying the condition]
\label{exa:pure-product-cyclic}
Use the same equations as the cyclic tower of
\cite[Section 4.2]{Nart}, but change the base field to $K=k((t))$,
where $k$ is algebraically closed of characteristic $2$ and $vt=1$:

\[
\alpha^2-\alpha=t^{-1},\qquad
\beta^2-\beta=t^{-1}\alpha,\qquad L=K(\alpha,\beta).
\]
Then $v\alpha=-1/2$ and $v\beta=-3/4$. The denominator of $v\beta$
forces $[K(\beta):K]\geq4$, so $L=K(\beta)$ has degree $4$ and is
totally ramified. The assignments
\[
\sigma(\alpha)=\alpha+1,\qquad \sigma(\beta)=\beta+\alpha
\]
preserve the defining equations and satisfy
$\sigma^2(\beta)=\beta+1$. Thus $L/K$ is Galois with group
$\langle\sigma\rangle\simeq C_4$.
If $[K(c):K]<4$, then $vc\ne-3/4$, since otherwise
$e(K(c)/K)\geq4$. Hence $v(\beta-c)\leq v\beta=v(\beta-0)$.
The pair $(\beta,0)$ is distinguished, so $\beta$ is pure and the
product hypothesis holds with $H_1=\mathcal G$.
\end{Exa}

\begin{Exa}[A cyclic quartic extension not satisfying the condition]
\label{exa:no-pure-product}
Return to the original cyclic tower of \cite[Section 4.2]{Nart},
recalled in Example \ref{exa:AS-cyclic-four}:
\[
K=\bigcup_{n\geq0}\overline{\mathbb F_2}((t^{1/2^n})),\qquad
\alpha^2-\alpha=t^{-1},\qquad
\eta^2-\eta=t^{-1}\alpha,\qquad M=K(\alpha,\eta).
\]
Its Galois group is $C_4$, but the extension has depth $2$ by
\cite[Section 4.2]{Nart}. This is a statement about the extension:
every generator has depth greater than one. Thus $(M/K,v)$ is not pure.
Its unique nontrivial proper subgroup has order $2$ and is pure, but
products of proper subgroups remain in that subgroup. Therefore
${\rm Gal}(M/K)$ is not a product of pure subgroups.
Unlike the preceding variation, this example is immediate, with defect
$4$. Both valuations have rank one.
\end{Exa}

\section*{Declaration of generative AI and AI-assisted technologies in the manuscript preparation process}

During the preparation of this work, the author used OpenAI ChatGPT as an auxiliary tool for mathematical discussion, including exploring alternative proof strategies, as well as for improving the organization, clarity, and language of the manuscript and identifying potentially relevant references. Every argument, computation, citation, and mathematical conclusion appearing in the final manuscript was independently checked and approved by the author, who takes full responsibility for the content of the article.

\end{document}